\documentclass[12pt]{article}  

\usepackage{amssymb}
\usepackage{amsfonts}
\usepackage{amsthm}

\usepackage[left=2cm,top=3cm,right=2cm,bottom=2cm,bindingoffset=0.5cm]{geometry}

\usepackage[numbers]{natbib}  
\usepackage{xcolor}
\usepackage{amsmath,amsthm,amsfonts,amssymb,color}
\usepackage[matrix,arrow]{xy}
\usepackage{graphicx}
\usepackage[UKenglish]{babel}
\usepackage{multicol}
\usepackage{appendix}
\usepackage{epstopdf}
\usepackage{comment}

\usepackage{caption}
\usepackage{subcaption}

\usepackage{multirow} 
\usepackage{float}
\usepackage{longtable}

\newcommand{\Ste}{\text{Ste}}

\DeclareMathOperator\erf{erf}
\DeclareMathOperator\erfc{erfc}

\newtheorem{thm}{Theorem}[section]
\newtheorem{cor}[thm]{Corollary}

\theoremstyle{definition}

\theoremstyle{remark}
\newtheorem{rem}[thm]{Remark}

\numberwithin{equation}{section}

\usepackage{tikz}
\usepackage{pgfplots}
\pgfplotsset{compat=1.16}

\begin{document}
\title{Relationship among solutions for three-phase change problems with Robin, Dirichlet, and Neumann boundary conditions}

\author{
Julieta Bollati$^{1,2}$, María F. Natale$^1$, José A. Semitiel$^1$, Domingo A. Tarzia$^{1,2}$ \\
\small{$^{1}$ Depto de Matemática, FCE-Universidad Austral, 2000 Rosario, Argentina} \\
\small{$^{2}$ CONICET, Argentina}
}

\date{}  

\maketitle  

\begin{abstract}
This study investigates the melting process of a three-phase Stefan problem in a semi-infinite material, imposing a convective boundary condition at the fixed face. By employing a similarity-type transformation, the problem is reduced to a solvable form, yielding a unique explicit solution. The analysis uncovers significant equivalences among the solutions of three different three-phase Stefan problems: one with a Robin boundary condition, another with a Dirichlet boundary condition, and a third one with a Neumann boundary condition at the fixed face. These equivalences are established under the condition that the problem data satisfy a specific relationship, providing new insights into the behaviour of phase change problems under varying boundary conditions.
\end{abstract}


\section{Introduction}

Stefan problems are a significant area of study because they occur in various important engineering and industrial contexts. They are crucial for understanding phase transition phenomena, especially in scenarios involving heat transfer and processes of solidification or melting. The goal of Stefan problems is to describe the liquid and solid phases of a material undergoing a phase change and to identify the location of the sharp interface that separates these phases, known as the free boundary. Transient heat conduction issues that include one or more phase changes are found in a number of practical areas. Applications of Stefan-type problems include the solidification of binary alloys \cite{BrPlVa2021, Ru1971,  SoWiAl1982, VeCiTa2024}, continuous casting of steel \cite{BeCaDu2023}, and cryopreservation of cells \cite{DaWaByHe2020}. So many applications of phase-change processes can be seen in the books \cite{ AlSo1993, Cr1984, Gu2018, KoKr2020, Lu1991, Pe2020, SzTh1971, Vi1996}.

During the solidification or melting process, the material can be divided into three distinct regions: a solid region, a mushy zone where both solid and liquid phases coexist, and a liquid region. In the case of polymorphous materials such as metallic iron and silica, multiple crystalline forms exist in the solid phase, resulting in several free boundaries between different phases. For example, metallic iron has three main crystalline forms, while silica exists in several distinct forms like quartz, tridymite, and cristobalite under high pressure. When these polymorphous materials freeze or melt, various phases are separated by multiple moving interfaces \cite{Ch2000,Tao1979}. 

Under certain boundary conditions, it is possible to find similarity type solutions to multiphase Stefan problems. In particular, in \cite{We1955, Wi1978}, it was considered a $n$-phase Stefan problem for a semi-infinite material imposing a constant temperature at the fixed face. A similar study was carried out in \cite{SaTa1989} with a Neumann type condition at the fixed face. A multiphase implicit Stefan problem was studied in \cite{ZhZhSh2017} for a one-dimensional non-Darcy flow in a semi-infinite porous media. However, the existence of analytical solutions to phase-change problems is challenging to ascertain due to the inherent non-linearity of these issues. For the analysis of more complex scenarios, numerical methods seem to be highly efficient. In \cite{Ch2000}, a hybrid numerical method is employed that combines the Laplace transform technique, control-volume formulation, and Taylor series approximations. In \cite{ZhZhBu2015}, an approximate analytical solution is derived for a non-linear multiphase Stefan problem, and the accuracy of this approximate method is assessed by comparison with the available exact solution. 

Heat transfer in three-phase systems presents significant challenges that are critical for a variety of applications. For instance, in \cite{AnNiShSoStSa2024}, researchers developed an analytical solution for the time-dependent heat transfer equation that accounts for phase change. This solution enables a new numerical algorithm to analyze temperature and heat flux variations in a three-layer building wall under transient ambient conditions. Additionally, \cite{AbGa2021} examines a numerical simulation of a triplex-tube thermal energy system that combines multiple phase change materials with porous metal foam. Another relevant application is presented in \cite{OlSaMaTa2008}, where an analytical solution for coupled heat and mass transfer during the freezing of high-water-content materials is developed.

Moreover, a new numerical method for modelling phase change problems involving three phases: solid, liquid, and gas is presented in \cite{ThPa2023}. It focuses on simulating melting and solidification of phase change materials with variable density and thermophysical properties. The method accounts for free surface dynamics and density changes during the phase transition. It revisits the two-phase Stefan problem, which involves a density jump between phases, and proposes a way to incorporate kinetic energy changes into the Stefan condition.

A specific case of a three-phase system is the  three-phase Stefan problem that consists of the solidification of an alloy. An alloy undergoes at least two phase changes when it solidifies: one when the temperature falls below the liquid temperature and another when it falls below the solid temperature. In contrast to the case of a pure metal, there are now two free boundaries corresponding to the liquid and solid temperatures.

The aim of this work is twofold.  First, the existence and uniqueness of solution to a three-phase melting Stefan problem is established, specifically under a Robin type boundary condition at the fixed face  $x=0$. Second, the connections between this problem and those arising from the imposition of Dirichlet or Neumann boundary conditions at the fixed face, are explored.

The organization of this paper is as follows. In Section 2, we formally present 
 one-dimensional Stefan problems with different  boundary conditions  concerning the melting of a semi-infinite material in the region $x\geq 0$, which undergoes three-phase changes. In addition, we prove the existence and uniqueness of the solution by imposing a Robin-type condition at the fixed face $x=0$. Furthermore, we retrieve similarity solutions from the existing literature that correspond to cases in which Dirichlet and Neumann boundary conditions are applied at the fixed boundary $x=0$. Finally, these solutions will be employed in Section 3 to establish a relationship among them.

\section{Three-phase Stefan problems with different boundary conditions}

In this section, we focus on the analysis of   three Stefan problems involving three phases for the melting of a semi-infinite material $x\geq 0$, each characterized by different conditions at the fixed face $x=0$. The aim is to determine the temperature 
\begin{equation}
\Phi(x,t)=\left\lbrace \begin{array}{ccll}
\Phi_3(x,t) \quad & \text{if} \quad & 0<x<y_2(t), \; &t>0, \\
\Phi_2(x,t) \quad & \text{if} \quad & y_2(t)<x<y_1(t), \; &t>0, \\
\Phi_1(x,t) \quad & \text{if} \quad & y_1(t)<x, \; &t>0, 
\end{array} \right.
\end{equation}
and the free boundaries $x=y_i(t)$, $i=1,2$, $t>0$ that separates the three regions that satisfy
\begin{align}
&  \frac{\partial \Phi_3}{\partial t}=\alpha_3 \frac{\partial^2 \Phi_3}{\partial x^2},  &0<x<y_2(t), \quad t>0, \label{EcCalor-Fase3}\\
&  \frac{\partial \Phi_2}{\partial t}=\alpha_2 \frac{\partial^2 \Phi_2}{\partial x^2},  &y_2(t)<x<y_1(t), \quad t>0, \label{EcCalor-Fase2}\\
&  \frac{\partial \Phi_1}{\partial t}=\alpha_1 \frac{\partial^2 \Phi_1}{\partial x^2},  &x>y_1(t), \quad t>0, \label{EcCalor-Fase1}\\
& \Phi_3(y_2(t),t)=\Phi_2(y_2(t),t)=B, \qquad &t>0,\label{B}\\
& \Phi_2(y_1(t),t)=\Phi_1(y_1(t),t)=C, \qquad &t>0,\label{C}\\
&  \Phi_1(x,0)=\Phi_1(+\infty,t)=D, \qquad &t>0, \label{D}\\
&  k_2 \frac{\partial \Phi_2}{\partial x}(y_2(t),t)-  k_3\frac{\partial \Phi_3}{\partial x}(y_2(t),t)=\delta_2 \dot{y_2}(t), &t>0, \label{CondStefan2}\\
&  k_1 \frac{\partial \Phi_1}{\partial x}(y_1(t),t)-  k_2\frac{\partial \Phi_2}{\partial x}(y_1(t),t)=\delta_1 \dot{y_1}(t), &t>0, \label{CondStefan1}\\
& y_1(0)=y_2(0)=0, \label{w_1 y w_2 0}
\end{align}
where the positive constants \(\alpha_i = \frac{k_i}{\rho c_i}\), \(k_i\), and \(c_i\) represent the thermal diffusivity, thermal conductivity, and specific heat, respectively, for phase \(i = 1, 2, 3\), with \(\rho\) being the common mass density. It is assumed throughout the paper that
\begin{equation}\label{hipalpha}
\alpha_2>\alpha_3.
\end{equation}

The latent heat per unit volume used for passing from phase $i$ to $i+1$ for $i=1,2$ is $\delta_i=\rho \ell_i$ where $\ell_i>0$ represents the latent heat per unit mass. The phase change temperatures $B$ and $C$, and the initial temperature $D$ verify the condition
\begin{equation}
B>C>D.
\end{equation}

In the following subsections, we will apply different boundary conditions at the fixed face $x=0$. We begin by 
considering a convective boundary condition given by

\begin{equation}
    k_3 \frac{\partial \Phi_3}{\partial x}(0,t) = \frac{h_0}{\sqrt{t}}\left(\Phi_3(0,t) - A_\infty\right), \qquad t > 0,
    \label{convectiva}
\end{equation}
where $h_0 > 0$ is the coefficient that characterizes the heat transfer at the fixed face and $A_\infty > B$ is the bulk temperature. In this case, we  demonstrate the existence and uniqueness of a similarity-type solution.

Next, we present the similarity-type solution obtained by imposing a Dirichlet boundary condition at the fixed face, as described in \cite{Wi1978}, considering only three phases. This condition is given by
\begin{equation}
    \Phi_3(0,t) = A, \qquad t > 0,
    \label{dirichlet}
\end{equation}
with \(A > B\). Finally, we  recover the similarity-type solution for a Neumann boundary condition from \cite{SaTa1989} for the three-phase Stefan problem, given by 
\begin{equation}
    k_3 \frac{\partial \Phi_3}{\partial x}(0,t) = -\frac{q_0}{\sqrt{t}}, \qquad t > 0,
    \label{neumann}
\end{equation}
where $q_0 > 0$.

\subsection{Existence and uniqueness of solution  by imposing a convective condition at the fixed face}
We propose a similarity-type solution to the problem  \eqref{EcCalor-Fase3}-\eqref{w_1 y w_2 0} and \eqref{convectiva} represented in the following manner:
\begin{align}
&  v_3(x,t)=A_3+B_3 \erf\left(\tfrac{x}{2\sqrt{\alpha_3 t}}\right), \qquad &0<x<w_2(t), \quad t>0, \label{v3}\\
&  v_2(x,t)=A_2+B_2\erf\left(\tfrac{x}{2\sqrt{\alpha_2 t}}\right), \qquad &w_2(t)<x<w_1(t), \quad t>0, \label{v2}\\
&  v_1(x,t)=A_1+B_1 \erf\left(\tfrac{x}{2\sqrt{\alpha_1 t}}\right), \qquad &x>w_1(t), \quad t>0, \label{v1}\\
&  w_2(t)=2\xi_2 \sqrt{\alpha_1t}, \quad &t>0, \label{w2}\\
&  w_1(t)=2\xi_1 \sqrt{\alpha_1t}, \quad &t>0, \label{w1}
\end{align}
where $A_i$ and $B_i$ are unknown constants to be determined for $i=1,2,3$, and $\xi_1$ and $\xi_2$ are positive dimensionless parameters that characterizes the free boundaries and must also be determined.

From  conditions (\ref{B})-(\ref{D}), the expressions for the temperatures in the three phases become:
\begin{align}
&  v_3(x,t)=\tfrac{\frac{Bk_3}{h_0 \sqrt{\pi \alpha_3}}+A_{\infty}\erf\left(\xi_2 \sqrt{\frac{\alpha_1}{\alpha_3}}\right)-(A_{\infty}-B) \erf\left(\frac{x}{2\sqrt{\alpha_3 t}}\right)}{\frac{k_3}{h_0 \sqrt{\pi \alpha_3}}+\erf \left(\xi_2 \sqrt{\frac{\alpha_1}{\alpha_3}}\right)}, & 0<x<w_2(t),\;t>0, \label{v3bis}\\
&  v_2(x,t)=\tfrac{-C \erf\left(\xi_2 \sqrt{\frac{\alpha_1}{\alpha_2}}\right)+B \erf\left(\xi_1 \sqrt{\frac{\alpha_1}{\alpha_2}}\right)-(B-C) \erf\left(\frac{x}{2\sqrt{\alpha_2 t}}\right)}{\erf\left(\xi_1 \sqrt{\frac{\alpha_1}{\alpha_2}}\right)-\erf\left(\xi_2 \sqrt{\frac{\alpha_1}{\alpha_2}}\right)} , & w_2(t)<x<w_1(t),\;t>0, \label{v2bis}\\
&  v_1(x,t)=\tfrac{C\left(1-\erf\left(\frac{x}{2\sqrt{\alpha_1 t}}\right)\right)+D\left(                       \erf\left(\frac{x}{2\sqrt{\alpha_1 t}}\right)-\erf(\xi_1) \right)}{\erfc(\xi_1)}, & x>w_1(t),\;t>0, \label{v1bis}
\end{align}
where $\erf$ and $\erfc$ denote the error function and the complementary error function, respectively.
The conditions \eqref{CondStefan2} and (\ref{CondStefan1})  are satisfied if $\xi_1$ and $\xi_2$ fulfill the following equalities:
\begin{align}
&  \xi_2=\tfrac{k_3}{\delta_2 \sqrt{\pi \alpha_1 \alpha_3}}\tfrac{(A_{\infty}-B)\exp\left(-\xi_2^2 \frac{\alpha_1}{\alpha_3}\right)}{\frac{k_3}{h_0 \sqrt{\pi \alpha_3}}+\erf\left(\xi_2 \sqrt{\frac{\alpha_1}{\alpha_3}}\right)}-\tfrac{k_2}{\delta_2 \sqrt{\pi \alpha_1 \alpha_2}}\tfrac{(B-C) \exp\left(-\xi_2 ^2 \frac{\alpha_1}{\alpha_2}\right)}{\erf\left(\xi_1 \sqrt{\frac{\alpha_1}{\alpha_2}}\right)-\erf\left(\xi_2 \sqrt{\frac{\alpha_1}{\alpha_2}}\right)}, \label{xi2}\\
& \xi_1=-\tfrac{k_1}{\delta_1 \alpha_1 \sqrt{\pi}}\tfrac{(C-D)\exp\left(-\xi_1^2 \right)}{\erfc\left(\xi_1\right)}-\tfrac{k_2}{\delta_1 \sqrt{\pi \alpha_1 \alpha_2}}\tfrac{(B-C) \exp\left(-\xi_1 ^2 \frac{\alpha_1}{\alpha_2}\right)}{\erf\left(\xi_1 \sqrt{\frac{\alpha_1}{\alpha_2}}\right)-\erf\left(\xi_2 \sqrt{\frac{\alpha_1}{\alpha_2}}\right)}. \label{xi1}
\end{align}

Expression (\ref{xi1}) can be rewritten as follows
 \begin{equation}
\erf\left(\xi_2 \sqrt{\tfrac{\alpha_1}{\alpha_2}}\right)=H(\xi_1),
\end{equation}
where the real function $H$ is defined by
\begin{equation}\label{H}
H(z)=\erf\left(z \sqrt{\tfrac{\alpha_1}{\alpha_2}}\right)-\tfrac{\Ste_2}{\sqrt{\pi}}\tfrac{\ell_2}{\ell_1}\sqrt{\tfrac{k_2 c_1}{k_1 c_2}}\tfrac{\exp\left(-z ^2 \frac{\alpha_1}{\alpha_2}\right)}{\varphi(z)}, \quad z\geq 0,
\end{equation}
\begin{equation}
\varphi(z)=z+\tfrac{\Ste_1}{\sqrt{\pi}}\tfrac{\exp\left(-z ^2\right)}{\erfc(z)}, \qquad z\geq 0,
\end{equation}
and  the Stefan numbers are defined by:
\begin{equation} \label{Ste}
\Ste_1=\tfrac{c_1(C-D)}{\ell_1}\qquad,\qquad \Ste_2=\tfrac{c_2(B-C)}{\ell_2}.
\end{equation}

Taking into account that $H$ is an increasing function that satisfies
$$H(0)=-\tfrac{\Ste_2}{\Ste_2}\tfrac{\ell_2}{\ell_1}\sqrt{\tfrac{k_2 c_1}{k_1 c_2}}<0, \qquad H(+\infty)=1,$$
then, there exists a unique $z_0>0$ such that 
\begin{equation}\label{z0}
z_0=H^{-1}(0).
\end{equation}
Therefore 
\begin{equation}\label{xi2-convectivo}
\xi_2=\sqrt{\tfrac{\alpha_1}{\alpha_2}}\erf^{-1}\left(H(\xi_1)\right), \quad \xi_1>z_0.
\end{equation}

Notice that 
\begin{equation*}
\erf\left(\xi_2 \sqrt{\tfrac{\alpha_1}{\alpha_2}}\right)=H(\xi_1)=\erf\left(\xi_1 \sqrt{\tfrac{\alpha_1}{\alpha_2}}\right)-\tfrac{\Ste_2}{\sqrt{\pi}}\tfrac{\ell_2}{\ell_1}\sqrt{\tfrac{k_2 c_1}{k_1 c_2}}\tfrac{\exp\left(-\xi_1^2 \tfrac{\alpha_1}{\alpha_2}\right)}{\varphi(\xi_1)}<\erf\left(\xi_1 \sqrt{\tfrac{\alpha_1}{\alpha_2}}\right),
\end{equation*}
then $\xi_2<\xi_1.$

Isolating $\left(\erf\left(\xi_1 \sqrt{\frac{\alpha_1}{\alpha_2}}\right)-\erf\left(\xi_2 \sqrt{\frac{\alpha_1}{\alpha_2}}\right)\right)^{-1}$
from \eqref{xi2} and \eqref{xi1} we obtain that $\xi_1$ must be a solution to the following equation
\begin{equation}
Q(z)=U(z),\qquad z>z_0, \label{ecxi1}
\end{equation}
where
\begin{equation}\label{Q}
Q(z)=\frac{\ell_1}{\ell_2}\varphi(z)\exp\left(z^2 \tfrac{\alpha_1}{\alpha_2}\right), \qquad z\geq 0,
\end{equation}
\begin{equation}\label{U}
U(z)=T\left(\sqrt{\tfrac{\alpha_2}{\alpha_1}}\erf^{-1}(H(z))\right), \qquad z>z_0.
\end{equation}
and
\begin{equation}\label{T}
T(z)=\tfrac{\Ste_2}{\sqrt{\pi}c_2} \tfrac{A_{\infty}-B}{B-C} \sqrt{\tfrac{k_3 c_1 c_3}{k_1}} \tfrac{\exp\left(-z^2 \alpha_1\left(\frac{1}{\alpha_3}-\frac{1}{\alpha_2}\right)\right)}{\frac{k_3}{h_0 \sqrt{\pi \alpha_3}}+\erf\left( z \sqrt{\frac{\alpha_1}{\alpha_3}}\right)}-z \exp\left(z^2 \tfrac{\alpha_1}{\alpha_2}\right), \qquad z>z_0.
\end{equation}

If $\alpha_2>\alpha_3$ we obtain that $U$ is a strictly decreasing function. From the fact that  $U(z_0)$ is a positive constant, $U(+\infty)=-\infty$ and $Q$ is strictly increasing function such that $Q(0)=\frac{\ell_1}{\ell_2}\frac{\Ste_1}{\sqrt{\pi}}$, $Q(+\infty)=+\infty$, one can infer that the solution $\xi_1$  to the equation (\ref{ecxi1}) is, indeed, unique in $(z_0,+\infty)$ if and only if $U(z_0)>Q(z_0)$. Since we know that $H(z_0)=0$, this inequality is equivalent to the following condition on the parameters of the problem 
\begin{equation}
(A_{\infty}-B)\sqrt{\tfrac{c_1 c_3\alpha_3}{k_1 k_3}}h_0 > \ell_1 \varphi(z_0) \exp\left(z_0^2 \tfrac{\alpha_1}{\alpha_2}\right) .
\end{equation} \label{hip1}

Moreover, from \eqref{H}, we can write the previous inequality in an equivalent way
\begin{equation}
(A_{\infty}-B) h_0 \sqrt{\tfrac{c_3\alpha_3}{k_3}}>\sqrt{\tfrac{k_2c_2}{\pi}}(B-C)\tfrac{1}{\erf\left(z_0 \sqrt{\frac{\alpha_1}{\alpha_2}}\right)},
\end{equation}
or else
\begin{equation}
h_0>\tfrac{B-C}{A_{\infty}-B} \sqrt{\tfrac{k_2k_3c_2}{\pi c_3\alpha_3}}\tfrac{1}{\erf\left(z_0 \sqrt{\tfrac{\alpha_1}{\alpha_2}}\right)}.
\end{equation}

The previous analysis leads to the following theorem
\begin{thm}\label{Teo:ExistenciaConvectivo}
Assuming   $h_0>h_2$ with 
\begin{equation}\label{h2}
h_2=\tfrac{B-C}{A_{\infty}-B} \sqrt{\tfrac{k_2k_3c_2}{\pi c_3\alpha_3}}\tfrac{1}{\erf\left(z_0 \sqrt{\tfrac{\alpha_1}{\alpha_2}}\right)},
\end{equation}
there exists a unique solution to the problem \eqref{EcCalor-Fase3}--\eqref{w_1 y w_2 0} and \eqref{convectiva}. The temperature $v_i$  in each phase, for $i=1,2,3$,  is described by \eqref{v3bis}, \eqref{v2bis} and \eqref{v1bis}, respectively. The free boundaries $w_2$ and $w_1$, given by \eqref{w2} and \eqref{w1}, are characterized by the dimensionless parameters $\xi_1$ and $\xi_2$. The parameter $\xi_2$ is defined by  \eqref{xi2-convectivo}, while $\xi_1$ is the unique solution to equation \eqref{ecxi1}.  
\end{thm}

\begin{rem}
In \cite{Ta2017}, a two-phase Stefan problem with a convective condition at the fixed face was studied. It was shown that for a unique solution to exist, the coefficient \( h_0 \) must satisfy the following inequality:

\[
h_0 > h_1 := \frac{k_1}{\sqrt{\pi \alpha_1}} \cdot \frac{C-D}{A_\infty - C}.
\]

In the case of three phases, we have proven that \( h_0 > h_2 \), and it is easy to see that the following relation holds:

\[
h_0 > h_2 > h_1.
\]

If the coefficient \( h_0 \) satisfies the following conditions:

\begin{enumerate}
    \item[i)] \( 0 < h_0 \leq h_1 \), then the problem defined by \eqref{EcCalor-Fase3}-\eqref{w_1 y w_2 0} and \eqref{convectiva} becomes a classical heat transfer problem for the initial solid phase. 
    \item[ii)] \( h_1 < h_0 \leq h_2 \), then the problem defined by \eqref{EcCalor-Fase3}-\eqref{w_1 y w_2 0} and \eqref{convectiva} becomes a two-phase Stefan problem. 
    \item[iii)] \( h_0 > h_2 \), then the problem defined by \eqref{EcCalor-Fase3}-\eqref{w_1 y w_2 0} and \eqref{convectiva} is a three-phase Stefan problem, whose unique solution is given by Theorem \ref{Teo:ExistenciaConvectivo}.
\end{enumerate}
\end{rem}

\begin{rem} In Figure \ref{MapaColores-convectivo}, we plot a colour map of the temperature $$v(x,t)=\left\lbrace \begin{array}{ccll}
v_3(x,t) \quad & \text{if} \quad & 0<x<w_2(t), \; &t>0, \\
v_2(x,t) \quad & \text{if} \quad & w_2(t)<x<w_1(t), \; &t>0, \\
v_1(x,t) \quad & \text{if} \quad & w_1(t)<x, \; &t>0, 
\end{array} \right.$$ described by \eqref{v3bis}-\eqref{v1bis} and the free boundaries $x=w_2(t)$ and $x=w_1(t)$ given by \eqref{w2} and \eqref{w1}, respectively, considering  the following parameters: $h_0=100 \; \tfrac{kg}{Ks^{{5/2}}}$, $A_\infty=334\;K$, $B=328\;K$, $C=324\;K$, $D=320\;K$, 
$k_1=0.2\; \tfrac{W}{mK}$,
$k_2=0.2\; \tfrac{W}{mK}$,
$k_3=0.2 \; \tfrac{W}{mK}$,
$c_1=2  \; \tfrac{J}{kg K}$,
$c_2=2 \; \tfrac{J}{kg K}$,
$c_3=2 \; \tfrac{J}{kg K}$,
$\rho=770\; \tfrac{kg}{m^3}$,
$\ell_1=160\; \tfrac{J}{kg}$,
$\ell_2=150\; \tfrac{J}{kg}$.
\end{rem}

\begin{figure}[h!!!!!!!!!!!!]
\begin{center}
\includegraphics[scale=0.27]{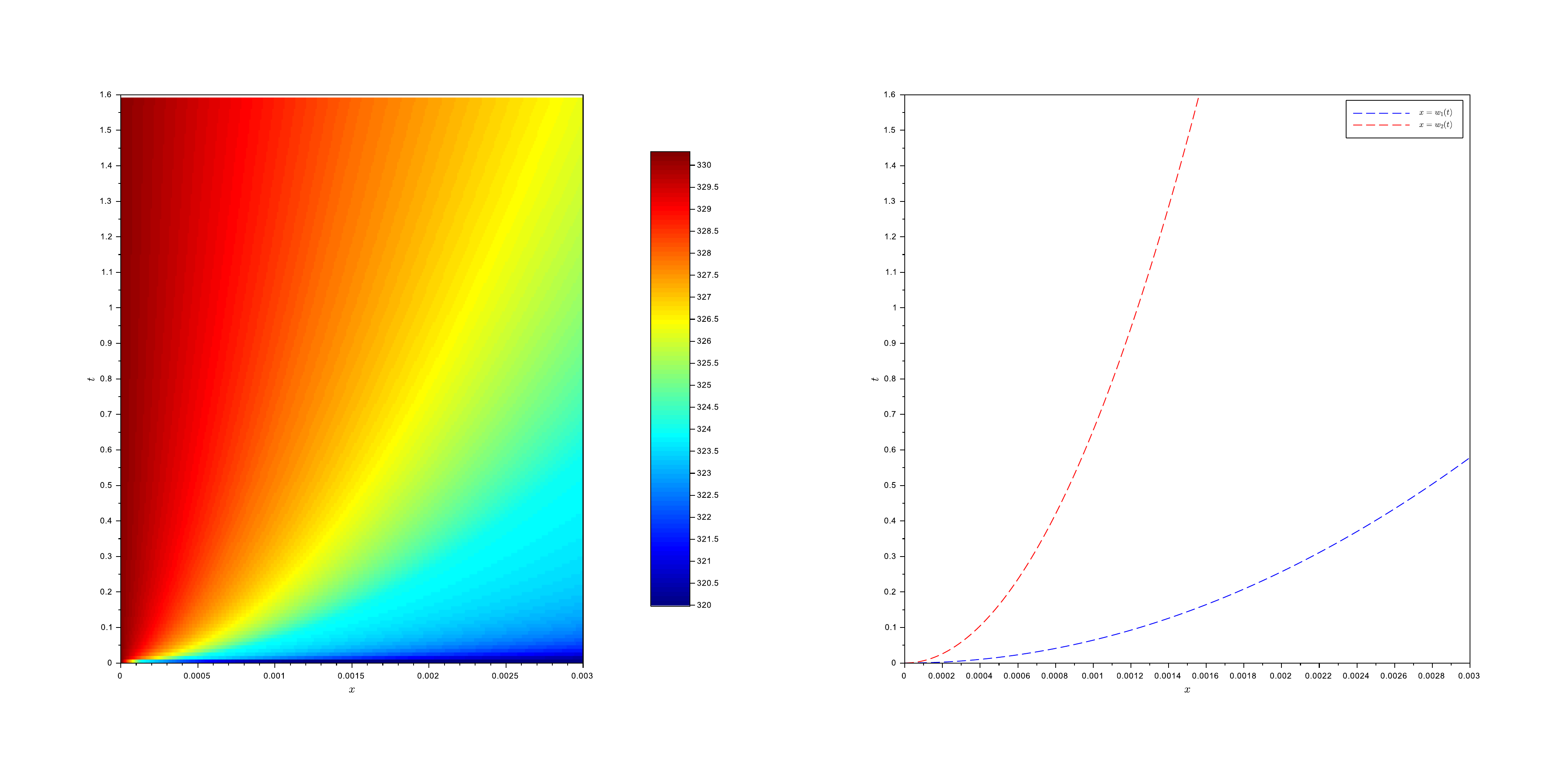}
\caption{Colour map of the temperature and the free boundaries of the three phase Stefan problem with a Robin type condition defined by \eqref{EcCalor-Fase3}--\eqref{w_1 y w_2 0} and \eqref{convectiva}.}
\label{MapaColores-convectivo}
\end{center}
\end{figure}

\subsection{Similarity-type solutions with a temperature and a flux condition at the fixed face}
The aim of this section is to connect the previously studied three-phase Stefan problem with the problems that involve Dirichlet and Neumann boundary conditions. These problems have been previously studied in the literature for the more general case of $n$ phases. Below, we present existence and uniqueness results for $n=3$.

From \citep{Wi1978} we can express the solution to the three-phase Stefan problem with a Dirichlet type condition at the fixed face $x=0$:

\begin{thm}\label{CasoTemperatura}
Under the assumption 
\begin{equation}\label{AB}
A>B,
\end{equation}
the problem  defined by \eqref{EcCalor-Fase3}-\eqref{w_1 y w_2 0} and \eqref{dirichlet} has  a unique solution given by

\begin{align}
&u_3(x,t)=A\tfrac{\erf\left(\mu_2  \sqrt{\tfrac{\alpha_1}{\alpha_3}}\right)-\erf\left(\tfrac{x}{2\sqrt{\alpha_3 t}}\right)}{\erf\left(\mu_2  \sqrt{\tfrac{\alpha_1}{\alpha_3}}\right)}+B\tfrac{\erf\left(\tfrac{x}{2\sqrt{\alpha_3 t}}\right)}{\erf\left(\mu_2  \sqrt{\tfrac{\alpha_1}{\alpha_3}}\right)}, & 0<x<r_2(t),\; t>0, \label{u3}\\
&u_2(x,t) = \tfrac{(C-B)\erf\left(\tfrac{x}{2\sqrt{\alpha_2 t}}\right)-C\erf\left(\mu_2 \sqrt{\tfrac{\alpha_1}{\alpha_2}}\right) + B\erf\left(\mu_1 \sqrt{\tfrac{\alpha_1}{\alpha_2}}\right)}{\erf\left(\mu_1 \sqrt{\tfrac{\alpha_1}{\alpha_2}}\right) - \erf\left(\mu_2 \sqrt{\tfrac{\alpha_1}{\alpha_2}}\right)}
,  & r_2(t)<x<r_1(t),\; t>0, \label{u2}\\
& u_1(x,t)=D+(C-D) \frac{\erfc\left(\frac{x}{2\sqrt{\alpha_1 t}}\right)}{\erfc(\mu_1)}, & x>r_1(t),\; t>0,\label{u1}\\
&r_2(t)= 2 \mu_2 \sqrt{\alpha_1 t}, & t>0,\label{Fron-mu2}\\
&r_1(t)= 2 \mu_1 \sqrt{\alpha_1 t}, & t>0,\label{Fron-mu1}
\end{align}
where
\begin{equation}
\mu_2=\sqrt{\tfrac{\alpha_2}{\alpha_1}}\erf^{-1}\left( H(\mu_1)\right), \label{mu2}
\end{equation}
and $\mu_1$ is the unique solution to 
\begin{equation}\label{QV}
Q(z)=V\left(  \sqrt{\tfrac{\alpha_2}{\alpha_1}}\erf^{-1}\left( H(z)\right)\right), \qquad z > z_0,
\end{equation}
where $z_0=H^{-1}(0)$ with $H$ defined by \eqref{H}, $Q$ is given by \eqref{Q}, and 
\begin{equation}\label{V}
V(z)=\tfrac{A-B}{\ell_2}\sqrt{\tfrac{c_1c_3k_3}{\pi k_1}}\tfrac{\exp\left(-z^2\left( \tfrac{\alpha_1}{\alpha_3}-\tfrac{\alpha_1}{\alpha_2} \right) \right)}{\erf\left(z \sqrt{\tfrac{\alpha_1}{\alpha_3}}\right)}-z\exp\left(z^2 \tfrac{\alpha_1}{\alpha_2}\right), \qquad z > 0.
\end{equation}
\end{thm}

\begin{rem} In Figure \ref{MapaColores-temperatura}, we plot a colour map of the temperature $$u(x,t)=\left\lbrace \begin{array}{ccll}
u_3(x,t) \quad & \text{if} \quad & 0<x<r_2(t), \; &t>0, \\
u_2(x,t) \quad & \text{if} \quad & r_2(t)<x<r_1(t), \; &t>0, \\
u_1(x,t) \quad & \text{if} \quad & r_1(t)<x, \; &t>0, 
\end{array} \right.$$ described by \eqref{u3}-\eqref{u1} and the free boundaries $x=r_2(t)$ and $x=r_1(t)$ given by \eqref{Fron-mu2} and \eqref{Fron-mu1}, respectively, considering  the following parameters: $A=331\;K$, $B=328\;K$, $C=324\;K$, $D=320\;K$, 
$k_1=0.2\; \tfrac{W}{mK}$,
$k_2=0.2\; \tfrac{W}{mK}$,
$k_3=0.2 \; \tfrac{W}{mK}$,
$c_1=2  \; \tfrac{J}{kg K}$,
$c_2=2 \; \tfrac{J}{kg K}$,
$c_3=2 \; \tfrac{J}{kg K}$,
$\rho=770\; \tfrac{kg}{m^3}$,
$\ell_1=160\; \tfrac{J}{kg}$,
$\ell_2=150\; \tfrac{J}{kg}$.
\end{rem}
\begin{figure}[h!!!!!!!!!!!!]
\begin{center}
\includegraphics[scale=0.27]{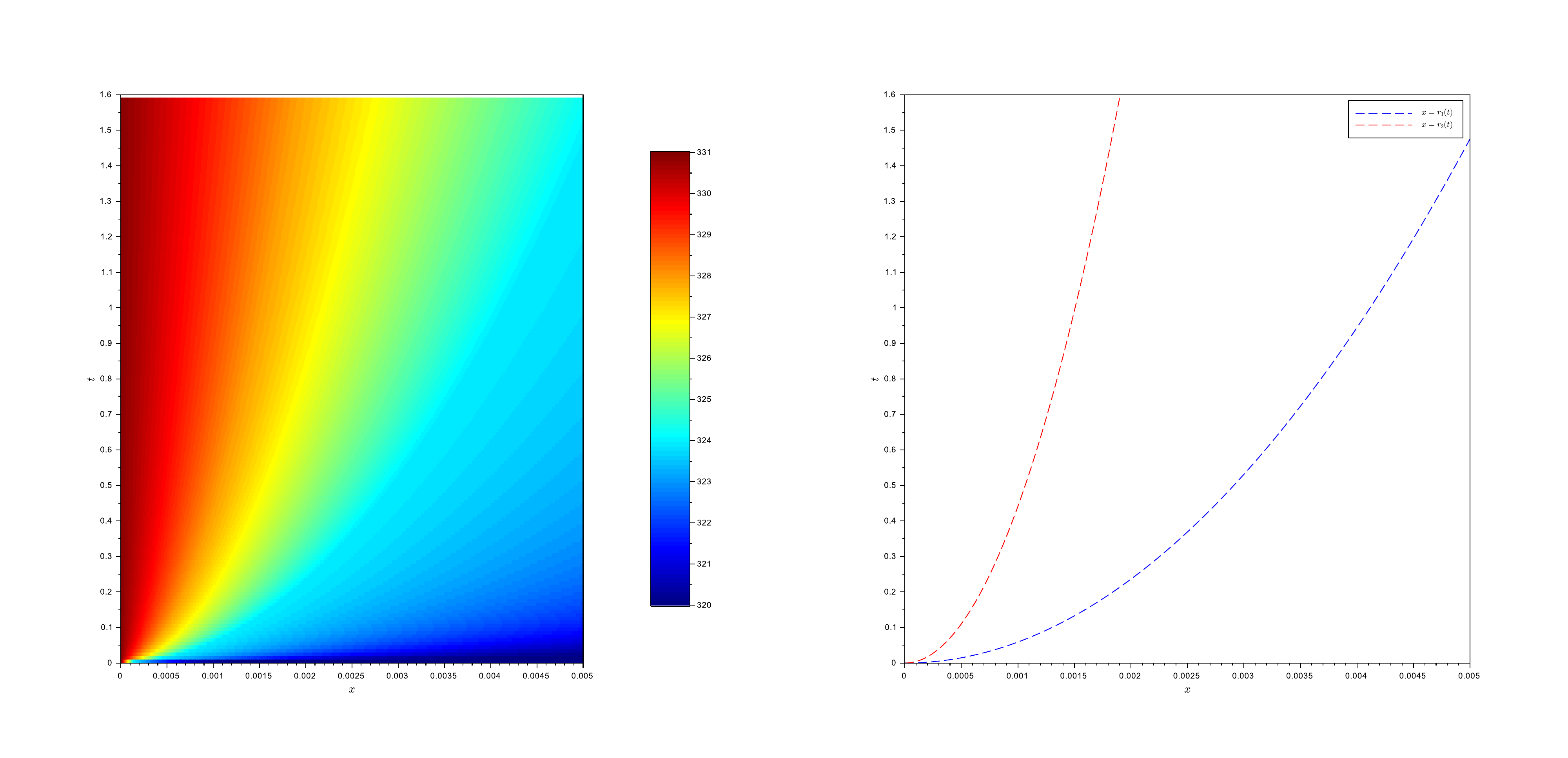}
\caption{Colour map of the temperature and the free boundaries of the three phase Stefan problem with a Dirichlet type condition defined by \eqref{EcCalor-Fase3}-\eqref{w_1 y w_2 0} and \eqref{dirichlet}.}
\label{MapaColores-temperatura}
\end{center}
\end{figure}

Based on \citep{SaTa1989}, we can represent the solution to the three-phase Stefan problem with a Neumann boundary condition at the fixed face $x=0$:
\begin{thm}\label{CasoFlujo}
If the following inequality holds
\begin{equation}\label{q0-q2}
q_0>  q_2:=\tfrac{k_2(B-C)}{\sqrt{\alpha_2\pi}\erf\left(z_0\sqrt{\tfrac{\alpha_1}{\alpha_2}}\right)},
\end{equation}
where $z_0$ is defined by \eqref{z0}, then the problem \eqref{EcCalor-Fase3}-\eqref{w_1 y w_2 0} and \eqref{neumann} has a unique solution  given by 
\begin{align}
& \theta_3(x,t)=B+\tfrac{q_0 \sqrt{\pi \alpha_3}}{k_3} \left(\erf\left(\lambda_2  \sqrt{\tfrac{\alpha_1}{\alpha_3}}\right)-\erf\left(\tfrac{x}{2\sqrt{\alpha_3 t}}\right)\right),& 0<x<s_2(t),\; t>0, \label{theta3}\\
&  \theta_2(x,t)=C+(B-C)\tfrac{\erf\left(\lambda_1  \sqrt{\tfrac{\alpha_1}{\alpha_2}}\right)-\erf\left(\tfrac{x}{2\sqrt{\alpha_2 t}}\right)}{\erf\left(\lambda_1  \sqrt{\tfrac{\alpha_1}{\alpha_2}}\right)-\erf\left(\lambda_2  \sqrt{\tfrac{\alpha_1}{\alpha_2}}\right)}, & s_2(t)<x<s_1(t),\; t>0, \label{theta2}\\
&  \theta_1(x,t)=D+(C-D) \frac{\erfc\left(\frac{x}{2\sqrt{\alpha_1 t}}\right)}{\erfc(\lambda_1)} ,& x>s_1(t),\;t>0.\label{theta1}\\
& s_2(t)= 2 \lambda_2 \sqrt{\alpha_1 t}, & t>0,\label{Fron-lambda2}\\
& s_1(t)= 2 \lambda_1 \sqrt{\alpha_1 t}, & t>0,\label{Fron-lambda1}
\end{align}
 where
\begin{equation}
\lambda_2=\sqrt{\tfrac{\alpha_2}{\alpha_1}}\erf^{-1}\left( H(\lambda_1)\right), \label{lambda2}
\end{equation}
and $\lambda_1$ is the unique solution to 
\begin{equation}\label{lambda1}
Q(z)=P\left(\sqrt{\tfrac{\alpha_2}{\alpha_1}}\erf^{-1}\left(H(z)\right)\right), \qquad z \geq 0, 
\end{equation}
with $Q$  given by \eqref{Q}, and
\begin{equation}\label{P}
P(z)=\exp\left(z^2\tfrac{\alpha_1}{\alpha_2}\right)\left[-z+\tfrac{q_0}{\ell_2}\sqrt{\tfrac{c_1}{\rho k_1}} \exp\left(-z^2\tfrac{\alpha_1}{\alpha_3}\right)\right],\qquad z \geq 0.
\end{equation}
\end{thm}

\begin{rem}
In  \cite{Ta1981-1982}, a two-phase Stefan problem with a Neumann condition at the fixed face was studied. It was proved that for a unique solution to exist,  $q_0$ must satisfy the following inequality $q_0>q_1:=\tfrac{k_1(C-D)}{\sqrt{\pi \alpha_1}}$.
In the case of three phases, we have proven that $q_0>q_2$  and it is easy to see that the following relation is satisfied:
$$q_0>q_2>q_1.$$

If the coefficient \(q_0 \) satisfies the following conditions:

\begin{enumerate}
    \item \( 0 < q_0 \leq q_1 \), then the problem defined by \eqref{EcCalor-Fase3}-\eqref{w_1 y w_2 0} and \eqref{neumann} reduces to a classical heat transfer problem for the initial solid phase.
    \item \( q_1 < q_0 \leq q_2 \), then the problem defined by \eqref{EcCalor-Fase3}-\eqref{w_1 y w_2 0} and \eqref{neumann} becomes a two-phase Stefan problem.
    \item \( q_0 > q_2 \), then the problem defined by \eqref{EcCalor-Fase3}-\eqref{w_1 y w_2 0} and \eqref{neumann} is a three-phase Stefan problem whose unique solution is given in Theorem \ref{CasoFlujo}.
\end{enumerate}
\end{rem}

\begin{rem} In Figure \ref{MapaColores-flujo}, we plot a colour map of the temperature $$\theta(x,t)=\left\lbrace \begin{array}{ccll}
\theta_3(x,t) \quad & \text{if} \quad & 0<x<s_2(t), \; &t>0, \\
\theta_2(x,t) \quad & \text{if} \quad & s_2(t)<x<s_1(t), \; &t>0, \\
\theta_1(x,t) \quad & \text{if} \quad & s_1(t)<x, \; &t>0, 
\end{array} \right.$$ described by \eqref{theta3}-\eqref{theta1} and the free boundaries $x=s_2(t)$ and $x=s_1(t)$ given by \eqref{Fron-lambda2} and \eqref{Fron-lambda1}, respectively, considering  the following parameters: $q_0=300 \; \tfrac{kg}{s^{{5/2}}}$, $B=328\;K$, $C=324\;K$, $D=320\;K$, 
$k_1=0.2\; \tfrac{W}{mK}$,
$k_2=0.2\; \tfrac{W}{mK}$,
$k_3=0.2 \; \tfrac{W}{mK}$,
$c_1=2  \; \tfrac{J}{kg K}$,
$c_2=2 \; \tfrac{J}{kg K}$,
$c_3=2 \; \tfrac{J}{kg K}$,
$\rho=770\; \tfrac{kg}{m^3}$,
$\ell_1=160\; \tfrac{J}{kg}$,
$\ell_2=150\; \tfrac{J}{kg}$.
\end{rem}
\begin{figure}[h!!!!!!!!!!!!]
\begin{center}
\includegraphics[scale=0.27]{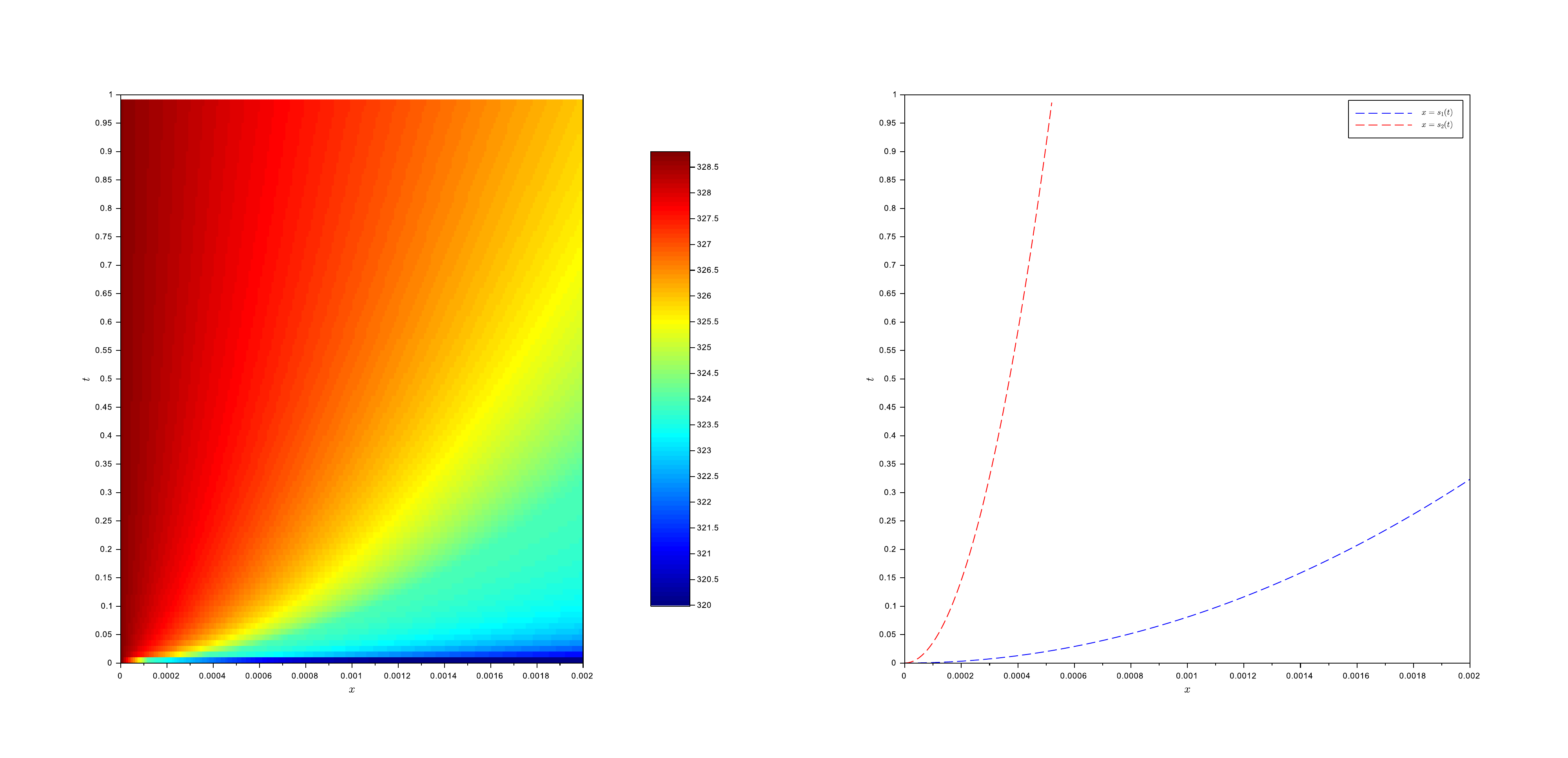}
\caption{Colour map of the temperature and the free boundaries of the three phase Stefan problem with a Neumann type condition defined by \eqref{EcCalor-Fase3}-\eqref{w_1 y w_2 0} and \eqref{neumann}.}
\label{MapaColores-flujo}
\end{center}
\end{figure}

\section{Relationship among problems}
From this point forward, we denote the problem governed by  \eqref{EcCalor-Fase3}-\eqref{w_1 y w_2 0} and \eqref{convectiva} as \textbf{(P1)}. If we substitute the Robin boundary condition \eqref{convectiva} with a temperature boundary condition, we obtain the problem defined by \eqref{EcCalor-Fase3}-\eqref{w_1 y w_2 0}, and \eqref{dirichlet}, which we will denote as \textbf{(P2)}. Similarly, we define problem \textbf{(P3)} as given by \eqref{EcCalor-Fase3}-\eqref{w_1 y w_2 0}, and \eqref{neumann}, which arises from replacing the condition \eqref{convectiva} with a Neumann boundary condition.

Having established our three problems, we will now demonstrate the equivalence among them. By equivalence, we refer to the condition in which, if the data of both problems satisfy a specific relationship, then they will yield the same solution. This will be shown through a detailed analysis of the boundary conditions and their implications for the solutions of the respective problems. We would like to emphasize that the results presented in this section are theoretical and applicable to all phase-change materials, with the possibility of experimental verification.

\subsection{Equivalence between problems \emph{\textbf{(P1)}} and \emph{\textbf{(P2)}}}

Linking a Stefan problem characterized by a temperature boundary condition to the  one with a convective boundary condition is essential for a deeper comprehension of heat transfer processes. Exploring how these conditions interact offers a broader perspective on the thermal dynamics of the system. 

In the subsequent theorem, we define the necessary relationships between the parameters of both problems to guarantee their equivalence.

\begin{thm}\label{Teor31} $ $
\begin{enumerate}
\item[\emph{a)}] Let $h_0$ and $A_\infty$ with $h_0 > h_2$ be  the given constants of the convective condition in the problem \emph{\textbf{(P1)}} where $h_2$ is defined by \eqref{h2}.  
If the following inequality holds:
\begin{equation}\label{condA}
\tfrac{\frac{Bk_3}{h_0\sqrt{\pi\alpha_3}}+A_\infty\erf\left(\xi_2\sqrt{\tfrac{\alpha_1}{\alpha_3}} \right)}{\frac{k_3}{h_0\sqrt{\pi \alpha_3}}+\erf\left(\xi_2\sqrt{\tfrac{\alpha_1}{\alpha_3}} \right)}>B,
\end{equation}
where  $\xi_2$ is given by \eqref{xi2-convectivo}, then the solution to problem \textbf{\emph{(P2)}} with
\begin{equation}\label{A}
A=\tfrac{\frac{Bk_3}{h_0\sqrt{\pi\alpha_3}}+A_\infty\erf\left(\xi_2\sqrt{\tfrac{\alpha_1}{\alpha_3}} \right)}{\frac{k_3}{h_0\sqrt{\pi \alpha_3}}+\erf\left(\xi_2\sqrt{\tfrac{\alpha_1}{\alpha_3}} \right)},
\end{equation}
 coincides with the solution to problem \textbf{\emph{(P1)}}.
\item[\emph{b)}] Let  $A>B$ be the data of the temperature boundary condition of the problem \textbf{\emph{(P2)}}. If the following inequality holds:
\begin{equation}\label{condh0}
\tfrac{A-B}{(A_{\infty}-A)\erf\left( \mu_2 \sqrt{\tfrac{\alpha_1}{\alpha_3}} \right)}>\tfrac{B-C}{A_{\infty}-B} \sqrt{\tfrac{k_2c_2}{k_3c_3}}\tfrac{1}{\erf\left(z_0 \sqrt{\frac{\alpha_1}{\alpha_2}}\right)},
\end{equation}
where $\mu_2$ is given by \eqref{mu2} and  $A_\infty>A$, then the solution to problem \textbf{\emph{(P1)}} with
\begin{equation}\label{h0equiv}
h_0=\tfrac{k_3}{\sqrt{\alpha_3 \pi}} \tfrac{A-B}{A_\infty-A} \tfrac{1}{\erf\left(\mu_2 \sqrt{\tfrac{\alpha_1}{\alpha_3}}  \right) },
\end{equation}
 coincides with the solution to problem \textbf{\emph{(P2)}}.
\end{enumerate}
\end{thm}

\begin{proof} $ $
\begin{enumerate}
\item[a)] Based on the solution to the problem \textbf{(P1)}, which involves the temperatures $v_3$, $v_2$, and $v_1$ and the free boundaries  $w_2$ and $w_1$ established  by Theorem \ref{Teo:ExistenciaConvectivo}, it is possible to determine the temperature at $x = 0$:

\begin{equation}
v_3(0,t)=\tfrac{\frac{Bk_3}{h_0 \sqrt{\pi \alpha_3}}+A_{\infty}\erf\left(\xi_2 \sqrt{\frac{\alpha_1}{\alpha_3}}\right)}{\frac{k_3}{h_0 \sqrt{\pi \alpha_3}}+\erf\left(\xi_2 \sqrt{\frac{\alpha_1}{\alpha_3}}\right)}, 
\end{equation}
then
\begin{equation}
v_3(0,t)-B=\tfrac{\left(A_{\infty}-B\right)\erf\left(\xi_2 \sqrt{\frac{\alpha_1}{\alpha_3}}\right)}{\frac{k_3}{h_0 \sqrt{\pi \alpha_3}}+\erf\left(\xi_2 \sqrt{\frac{\alpha_1}{\alpha_3}}\right)}>0, 
\end{equation}

Taking into account that $v_3(0,t)>B$ we can formulate the three-phase Stefan problem \textbf{(P2)} with a temperature condition at $x=0$ given by
$A=v_3(0,t)$ defined by \eqref{A}.

Then, for this data, we can rewrite the solution to problem \textbf{(P2)} as
$$ u_3(x,t)=\tfrac{\frac{Bk_3}{h_0 \sqrt{\pi \alpha_3}}+A_{\infty}\erf\left(\xi_2 \sqrt{\frac{\alpha_1}{\alpha_3}}\right)}{\frac{k_3}{h_0 \sqrt{\pi \alpha_3}}+\erf\left(\xi_2 \sqrt{\frac{\alpha_1}{\alpha_3}}\right)}\tfrac{\erf\left(\mu_2  \sqrt{\tfrac{\alpha_1}{\alpha_3}}\right)-\erf\left(\tfrac{x}{2\sqrt{\alpha_3 t}}\right)}{\erf\left(\mu_2  \sqrt{\tfrac{\alpha_1}{\alpha_3}}\right)}+B\tfrac{\erf\left(\tfrac{x}{2\sqrt{\alpha_3 t}}\right)}{\erf\left(\mu_2  \sqrt{\tfrac{\alpha_1}{\alpha_3}}\right)},\quad 0<x<r_2(t), \; t>0,$$
and  (\ref{u2})-(\ref{QV}).

Taking into account  (\ref{Ste}), \eqref{A}, and the equations  \eqref{mu2}, \eqref{QV}, it follows that the coefficients $\mu_1$ and $\mu_2$ constitute the unique solution to the following system of equations
\begin{equation}\label{sistemaP1P2n}
\left\lbrace
\begin{array}{llll}
& \tfrac{\ell_1}{\ell_2} \varphi(z_1) \exp\left(z_1^2 \tfrac{\alpha_1}{\alpha_2}\right) =\tfrac{\Ste_2}{c_2} \tfrac{A_{\infty} - B}{B - C} \sqrt{\tfrac{c_1 c_3 k_3}{k_1 \pi}} \tfrac{\erf\left(\xi_2 \sqrt{\tfrac{\alpha_1}{\alpha_3}}\right)} {\erf\left(z_2 \sqrt{\tfrac{\alpha_1}{\alpha_3}}\right)} \tfrac{\exp\left(-z_2^2\alpha_1 \left(\tfrac{1}{\alpha_3}-\tfrac{1}{\alpha_2}\right)\right)}{\frac{k_3}{h_0 \sqrt{\pi \alpha_3}}+\erf\left(\xi_2 \sqrt{\tfrac{\alpha_1}{\alpha_3}}\right)} - z_2 \exp\left(z_2^2 \tfrac{\alpha_1}{\alpha_2}\right), \\
& z_2  =\sqrt{\tfrac{\alpha_2}{\alpha_1}} \erf^{-1}(H(z_1)).
\end{array}
\right.
\end{equation}

Considering that \( \xi_2 \) is given by  \eqref{xi2} and that \( \xi_1 \) is the unique solution to \eqref{xi1}, it follows that $\xi_1$ and $\xi_2$ constitute a solution of the system \eqref{sistemaP1P2n}. By uniqueness, we obtain that \( \xi_1 = \mu_1 \) and \( \xi_2 = \mu_2 \).
From this fact, it follows immediately that $v_i(x,t)=u_i(x,t)$ for $i=1,2,3$.

\item[b)]  From the temperatures \( u_3 \), \( u_2 \), and \( u_1 \), along with the free boundaries \( r_2 \) and \( r_1 \) that represent the unique solution to problem \textbf{(P2)} as specified by \eqref{u3}-\eqref{V}, we obtain
$$
k_3 \frac{\partial u_3}{\partial x}(0,t) = \tfrac{-(A-B)k_3}{\sqrt{\pi \alpha_3} \, \erf\left( \mu_2 \sqrt{\tfrac{\alpha_1}{\alpha_3}} \right)}  \tfrac{1}{\sqrt{t}},
$$
and therefore we can compute the convective condition at the fixed face $x=0$ given by \eqref{convectiva} with $h_0$ given by \eqref{h0equiv}
for some value $A_\infty>A$ such that \eqref{condh0} holds. Then the unique solution to the  problem \textbf{(P1)} with $h_0$ defined by \eqref{h0equiv} is given by
\begin{equation}\label{v3-equiv}
v_3(x,t)=\tfrac{ \tfrac{B(A_\infty-A)}{A-B} \erf\left(\mu_2 \sqrt{\tfrac{\alpha_1}{\alpha_3}} \right) +A_\infty \erf\left(\xi_2 \sqrt{\tfrac{\alpha_1}{\alpha_3}} \right) -(A_\infty-B) \erf\left(\tfrac{x}{2\sqrt{\alpha_3 t}} \right) }{ \tfrac{A_\infty-A}{A-B} \erf\left(\mu_2 \sqrt{\tfrac{\alpha_1}{\alpha_3}} \right)  +\erf\left(\xi_2 \sqrt{\tfrac{\alpha_1}{\alpha_3}} \right)}, \quad 0<x<w_2(t),\; t>0,
\end{equation}
and \eqref{w2}-\eqref{w1} and \eqref{v2bis}-\eqref{v1bis}.
In addition, taking into account  (\ref{Ste}), \eqref{xi2-convectivo} and \eqref{ecxi1}, the coefficients  $\xi_1$ and $\xi_2$ constitute the unique solution to the following system of equations
\begin{equation}\label{sistemaP2P1}
\left\lbrace
\begin{array}{llll}
& \tfrac{\ell_1}{\ell_2} \varphi(z_1) \exp\left(z_1^2 \tfrac{\alpha_1}{\alpha_2}\right) =\tfrac{\Ste_2}{c_2}  \tfrac{A_{\infty} - B}{B - C} \sqrt{\tfrac{c_1 c_3 k_3}{k_1 \pi}} \tfrac{\exp\left(-z_2^2\alpha_1 \left(\tfrac{1}{\alpha_3}-\tfrac{1}{\alpha_2}\right)\right)}{\erf\left(z_2 \sqrt{\tfrac{\alpha_1}{\alpha_3}}\right)+\tfrac{A_\infty-A}{A-B}\erf\left(\mu_2 \sqrt{\tfrac{\alpha_1}{\alpha_3}}\right) } - z_2 \exp\left(z_2^2 \tfrac{\alpha_1}{\alpha_2}\right), \\
& z_2  =\sqrt{\tfrac{\alpha_2}{\alpha_1}} \erf^{-1}(H(z_1)).
\end{array}
\right.
\end{equation}
Considering that \( \mu_2 \) is given by  \eqref{mu2} and that \( \mu_1 \) is the unique solution to \eqref{QV}, it follows that $\mu_1$ and $\mu_2$ constitute a solution to \eqref{sistemaP2P1}. By uniqueness, we obtain that \( \mu_1 = \xi_1 \) and \( \mu_2 = \xi_2 \).
From this fact, it follows immediately that $u_i(x,t)=v_i(x,t)$ for $i=1,2,3$.
\end{enumerate}

\end{proof}

From the previous theorem, the equivalence of problems \textbf{(P1)} and \textbf{(P2)} arises under certain conditions regarding the data; therefore, the following relationships can be established. This insight highlights that, under specific circumstances, two distinct problems share a common solution framework, which can be crucial for simplifying analyses or applications in various fields.

\begin{cor} Let  $A>B$ be the data of the temperature boundary condition of the problem \textbf{\emph{(P2)}}. The coefficient $\mu_2$ that characterizes the free boundary $r_2$  given by \eqref{mu2} satisfies the following inequality:
\begin{equation}\label{inec-mu2}
\erf\left(\mu_2\sqrt{\tfrac{\alpha_1}{\alpha_3}}\right)<\sqrt{\tfrac{k_3c_3}{k_2c_2}}\tfrac{A-B}{B-C} \tfrac{A_\infty-B}{A_\infty-A}\erf\left( z_0\sqrt{\tfrac{\alpha_1}{\alpha_2}} \right), \qquad \forall A_\infty >A,
\end{equation}
where $z_0$ is given by \eqref{z0}.
\end{cor}
\begin{proof}
From part b) of the previous Theorem, it follows that $h_0$, as given by \eqref{h0equiv}, must satisfy $h_0 > h_2$, where \( h_2 \) is defined by \eqref{h2}. Therefore, the coefficients $\mu_2$ that characterize the interface $x = r_2(t)$ of the solution to the problem \textbf{(P2)}, must also satisfy the inequality \eqref{inec-mu2}.
\end{proof}

\begin{rem} The function defined by the right hand side in \eqref{inec-mu2} is a strictly decreasing function of the variable $A_\infty$. Then, taking the limit $A_\infty \to +\infty$ to the inequality \eqref{inec-mu2}, we obtain the following inequality:
\begin{equation}\label{inec-mu2-1}
\erf\left(\mu_2\sqrt{\tfrac{\alpha_1}{\alpha_3}}\right)<\sqrt{\tfrac{k_3c_3}{k_2c_2}}\tfrac{A-B}{B-C}\erf\left( z_0\sqrt{\tfrac{\alpha_1}{\alpha_2}} \right).
\end{equation}
\end{rem}

\begin{rem} The inequality \eqref{inec-mu2} holds a physical significance for the solution \eqref{u3}-\eqref{QV} when the parameters of the problem \textbf{(P2)} satisfy the inequality:
\begin{equation}\label{inec-mu2-2}
\sqrt{\tfrac{k_3c_3}{k_2c_2}}\tfrac{A-B}{B-C} \tfrac{A_\infty-B}{A_\infty-A} \erf\left( z_0\sqrt{\tfrac{\alpha_1}{\alpha_2}} \right)<1, \qquad  A_\infty >A>B>C>D.
\end{equation}
\end{rem}

\begin{rem} The coefficient $h_0$ given by \eqref{h0equiv} can be thought as a function that depends on $A$. If we consider the following parameters: $A_\infty=334\;K$, $B=328\;K$, $C=324\;K$, $D=320\;K$, 
$k_1=0.2\; \tfrac{W}{mK}$,
$k_2=0.2\; \tfrac{W}{mK}$,
$k_3=0.2 \; \tfrac{W}{mK}$,
$c_1=2  \; \tfrac{J}{kg K}$,
$c_2=2 \; \tfrac{J}{kg K}$,
$c_3=2 \; \tfrac{J}{kg K}$,
$\rho=770\; \tfrac{kg}{m^3}$,
$\ell_1=160\; \tfrac{J}{kg}$,
$\ell_2=150\; \tfrac{J}{kg}$,
we can plot \( h_0 \) as a function of \( A \).  Figure \ref{Fig1:ho(A)} clearly shows that \( h_0 \) is a strictly increasing function, with a vertical asymptote at \(A= A_\infty \).

\begin{figure}[h!!!]
\begin{center}
\includegraphics[scale=0.5]{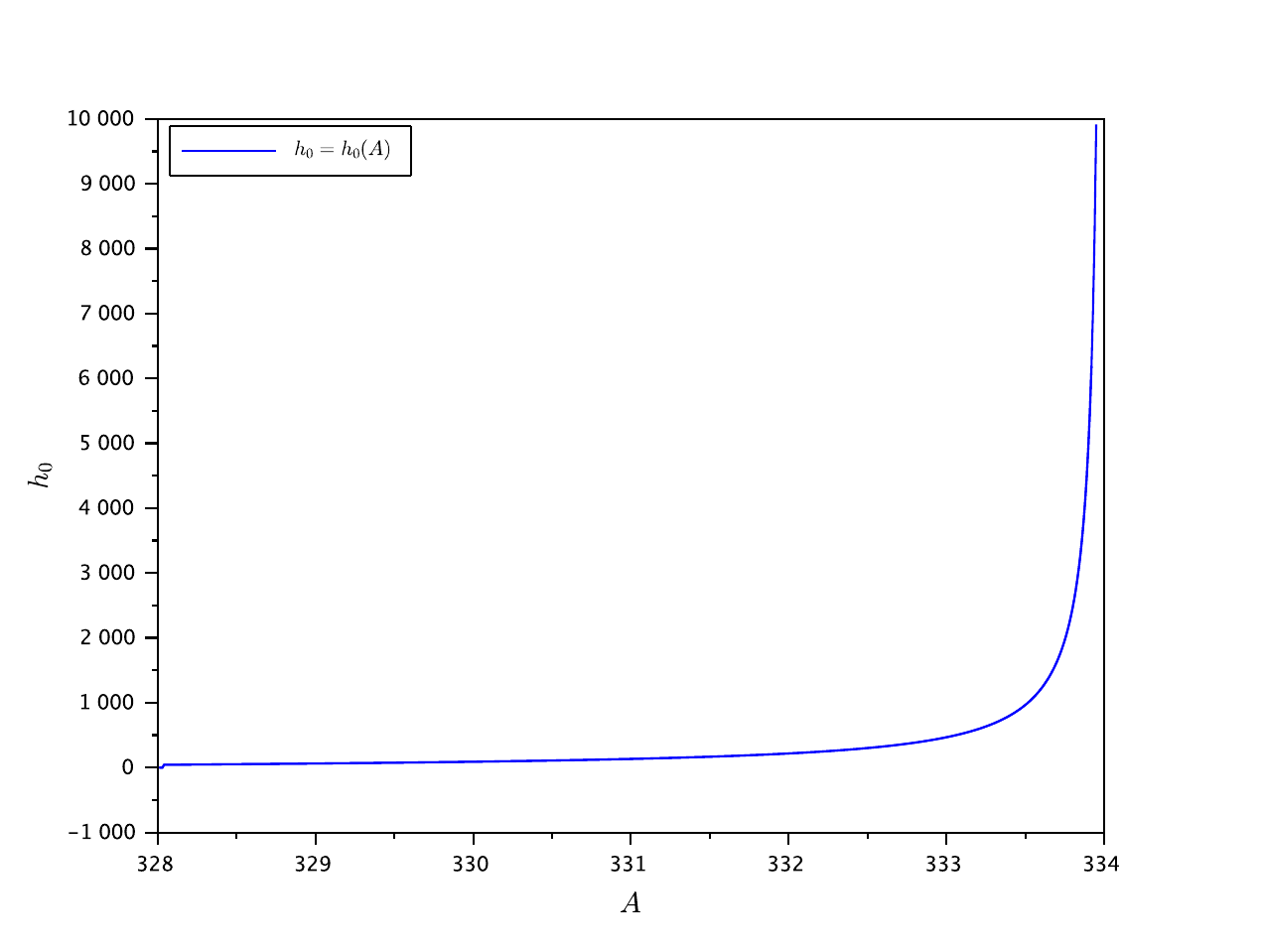}
\caption{Coefficient $h_0$ characterizing heat transfer under convective conditions as a function of 
$A$.}
\label{Fig1:ho(A)}
\end{center}
\end{figure}
\end{rem}

\begin{cor} Let $h_0$ and $A_\infty$ with $h_0>h_2$ be the data of the convective boundary condition of the problem \textbf{\emph{(P1)}}. The value $A$  that characterizes the temperature boundary condition at the fixed face $x=0$ to the problem  \textbf{\emph{(P2)}}, satisfies the following inequality:
\begin{equation}\label{inec-A}
B<A<A_\infty.
\end{equation}

\end{cor}
\begin{proof}
From part a) of Theorem \ref{Teor31}, it follows that $A$, as given by \eqref{condA},  satisfy $A-A_\infty<0$ and $A-B>0$. Therefore, $A$ must also satisfy the inequality \eqref{inec-A}.
\end{proof}

\begin{cor} The value $A=A(h_0,A_\infty)$ given by \eqref{condA} is an increasing function of $h_0$. 
\end{cor}
\begin{proof}

Notice that 
$$A(h_0, A_\infty) = B \frac{1 + A_\infty \frac{h_0 \sqrt{\pi \alpha_3}}{B k_3} \erf\left(\xi_2 \sqrt{\tfrac{\alpha_1}{\alpha_3}} \right)}{1 + \frac{h_0 \sqrt{\pi \alpha_3}}{k_3} \erf\left(\xi_2 \sqrt{\tfrac{\alpha_1}{\alpha_3}} \right)} = B \Psi\left(h_0 \erf\left(\xi_2 \sqrt{\tfrac{\alpha_1}{\alpha_3}} \right) \right)$$
where  $\Psi(z) = \frac{1 + \nu_1 z}{1 + \nu_2 z}$ with  $\nu_1 = \frac{\sqrt{\pi \alpha_3}}{k_3} \frac{A_\infty}{B}$ and $\nu_2 = \frac{\sqrt{\pi \alpha_3}}{k_3}.$
Taking into account that $A_\infty > B$, it follows that $\nu_1 - \nu_2 = \frac{\sqrt{\pi \alpha_3}}{k_3} \left(\frac{A_\infty}{B} - 1 \right) > 0$ and therefore  $\Psi'(z) = \frac{\nu_1 - \nu_2}{(1 + \nu_2 z)^2} > 0$ for all $z$. This means that $\Psi$ is an increasing function in $z$. 

In addition, notice that $\xi_2$ given by \eqref{xi2-convectivo} depends on $\xi_1$. In turn, $\xi_1$ is the unique solution to the equation \eqref{ecxi1}. On one hand, the function $T$ defined by \eqref{T} is an increasing function in $h_0$. As a consequence, by \eqref{U}, the function $U$ increases when $h_0$ becomes greater. On the other hand, the function $Q$ given by \eqref{Q} does not depend on $h_0$. Then it follows that the unique solution $\xi_1>z_0$ to the equation \eqref{ecxi1} increases in $h_0$. Consequently, $\xi_2$ also becomes an increasing function in $h_0$.

Putting all of the above together, it is easy to see that $A$ is an increasing function in $h_0$.

\end{proof}

\subsection{Equivalence between problems \emph{\textbf{(P2)}} and \emph{\textbf{(P3)}}}

Connecting a Stefan problem with a temperature boundary condition to the one with a flux boundary condition is crucial for understanding thermal phenomena. Analyzing their interaction provides a more comprehensive view of the thermal behavior of the system. Additionally, this relationship can reveal important mathematical properties of the underlying equations. 

In the following theorem, we establish the relationships to impose between the data of both problems in order to ensure their equivalence.

\begin{thm}\label{TeorP2P3} $ $
\begin{enumerate}
\item[\emph{a)}] Let $q_0>q_2$   be  the given constant of the flux condition of the problem \emph{\textbf{(P3)}}  where $q_2$ is defined by \eqref{q0-q2}. If we consider the problem \textbf{\emph{(P2)}} with a temperature boundary condition defined by
\begin{equation}\label{condA-relP3}
A=B+ \tfrac{q_0 \sqrt{\pi \alpha_3}}{k_3} \erf\left(\lambda_2 \sqrt{\tfrac{\alpha_1}{\alpha_3}} \right),
\end{equation}
where $\lambda_2$ is given by \eqref{lambda2}, then the solution to problem \textbf{\emph{(P2)}} coincides with the solution to problem \textbf{\emph{(P3)}}.
\item[\emph{b)}] Let  $A>B$  be the data of the temperature boundary condition of the problem \textbf{\emph{(P2)}}. If the following inequality holds:
\begin{equation}\label{condq-relP2}
\tfrac{k_3(A-B)}{\sqrt{\alpha_3}\erf\left( \mu_2 \sqrt{\tfrac{\alpha_1}{\alpha_3}} \right)}>\tfrac{k_2(B-C)}{\sqrt{\alpha_2}\erf\left(z_0\sqrt{\tfrac{\alpha_1}{\alpha_2}}\right)},
\end{equation}
where $\mu_2$ is given by \eqref{Fron-mu2},  then the solution to problem \textbf{\emph{(P3)}} with
\begin{equation}\label{q0(A)}
q_0=\tfrac{k_3}{\sqrt{\pi \alpha_3}} \tfrac{(A-B)}{\erf\left( \mu_2 \sqrt{\tfrac{\alpha_1}{\alpha_3}} \right)},
\end{equation}
coincides with the solution to problem \textbf{\emph{(P2)}}.
\end{enumerate}
\end{thm}

\begin{proof} $ $
\begin{enumerate}

\item[a)] Using the solution to the problem \textbf{(P3)}, which involves the temperatures $\theta_3$, $\theta_2$ and $\theta_1$, along with the free boundaries $s_2$ and $s_1$ established in Section 2.2, it is possible to determine the temperature at $x = 0$:
\begin{equation}
\theta_3(0,t)=B+ \tfrac{q_0 \sqrt{\pi \alpha_3}}{k_3} \erf\left(\lambda_2 \sqrt{\tfrac{\alpha_1}{\alpha_3}} \right).
\end{equation}
Taking into account that $\theta_3(0,t) > B$, we can define the three-phase problem \textbf{(P2)} with a temperature condition at $x=0$ specified as $A = \theta_3(0,t)$, i.e. $A$ is given by \eqref{condA-relP3}.

Then, for this data, we can rewrite the solution to problem \textbf{(P2)} as
$$ u_3(x,t)= \left(B+ \tfrac{q_0 \sqrt{\pi \alpha_3}\erf\left(\lambda_2 \sqrt{\tfrac{\alpha_1}{\alpha_3}} \right)}{k_3}  \right)\tfrac{\erf\left(\mu_2  \sqrt{\tfrac{\alpha_1}{\alpha_3}}\right)-\erf\left(\tfrac{x}{2\sqrt{\alpha_3 t}}\right)}{\erf\left(\mu_2  \sqrt{\tfrac{\alpha_1}{\alpha_3}}\right)}+B\tfrac{\erf\left(\tfrac{x}{2\sqrt{\alpha_3 t}}\right)}{\erf\left(\mu_2  \sqrt{\tfrac{\alpha_1}{\alpha_3}}\right)},\; 0<x<r_2(t), \; t>0,$$
and  (\ref{u2})-(\ref{QV}).

From the values given by (\ref{Ste}) and \eqref{condA-relP3}, and the equations  \eqref{mu2} and \eqref{QV}, it follows that the coefficients $\mu_1$ and $\mu_2$ constitute the unique solution to the following system of equations
\begin{equation}\label{sistemaP2P3}
\left\lbrace
\begin{array}{llll}
& \tfrac{\ell_1}{\ell_2} \varphi(z_1) \exp\left(z_1^2 \tfrac{\alpha_1}{\alpha_2}\right) =\frac{q_0}{\ell_2}\sqrt{\tfrac{c_1 }{\rho k_1}} \tfrac{\erf\left(\lambda_2 \sqrt{\tfrac{\alpha_1}{\alpha_3}}\right) \exp\left(-z_2^2\alpha_1 \left(\tfrac{1}{\alpha_3}-\tfrac{1}{\alpha_2}\right)\right)} {\erf\left(z_2 \sqrt{\tfrac{\alpha_1}{\alpha_3}}\right)}  - z_2 \exp\left(z_2^2 \tfrac{\alpha_1}{\alpha_2}\right), \\
& z_2  =\sqrt{\tfrac{\alpha_2}{\alpha_1}} \erf^{-1}(H(z_1)).
\end{array}
\right.
\end{equation}
Given that \( \lambda_2 \) is defined by \eqref{lambda2} and \( \lambda_1 \) is the unique solution to \eqref{lambda1}, we can conclude that \( \lambda_1 \) and \( \lambda_2 \) form a solution to the system \eqref{sistemaP2P3}. Due to uniqueness, it follows that \( \lambda_1 = \mu_1 \) and \( \lambda_2 = \mu_2 \). Consequently, we can immediately deduce that \( \theta_i(x,t) = u_i(x,t) \) for \( i = 1, 2, 3 \).

\item[b)] Using the temperatures \( u_3 \), \( u_2 \), and \( u_1 \), along with the free boundaries \( r_2 \) and \( r_1 \), which represent the unique solution to problem \textbf{(P2)} as defined by \eqref{u3}-\eqref{V}, we derive
$$
k_3 \frac{\partial u_3}{\partial x}(0,t) = \tfrac{-(A-B)k_3}{\sqrt{\pi \alpha_3} \, \erf\left( \mu_2 \sqrt{\tfrac{\alpha_1}{\alpha_3}} \right)}  \tfrac{1}{\sqrt{t}},
$$
and therefore we can compute the flux condition at the fixed face $x=0$ given by \eqref{neumann}. Since \eqref{condq-relP2} holds, the unique solution to the  problem \textbf{(P3)} with $q_0$ defined by \eqref{q0(A)} is given by 
\begin{equation}\label{theta3-equiv}
\theta_3(x,t)=B+\tfrac{A-B}{\erf\left(\mu_2 \sqrt{\tfrac{\alpha_1}{\alpha_3}} \right)} \left( \erf\left( \lambda_2 \sqrt{\tfrac{\alpha_1}{\alpha_3}}\right) -\erf\left( \tfrac{x}{2\sqrt{\alpha_3 t}}\right) \right), \quad 0<x<s_2(t),\; t>0,
\end{equation}
and \eqref{theta2}-\eqref{lambda1}.
In addition, taking into account (\ref{Ste}),  \eqref{lambda2} and \eqref{lambda1},  the coefficients  $\lambda_1$ and $\lambda_2$ constitute the unique solution to the following system of equations
\begin{equation}\label{sistemaP1P2}
\left\lbrace
\begin{array}{llll}
& \tfrac{\ell_1}{\ell_2} \varphi(z_1) \exp\left(z_1^2 \tfrac{\alpha_1}{\alpha_2}\right) =\tfrac{A-B}{\ell_2} \sqrt{\tfrac{c_1 c_3 k_3}{k_1 \pi}} \tfrac{ \exp\left( -z_2^2 \alpha_1 \left(\tfrac{1}{\alpha_3}-\tfrac{1}{\alpha_2} \right) \right)}{\erf\left(\mu_2  \sqrt{\tfrac{\alpha_1}{\alpha_3}}\right)} - z_2 \exp\left(z_2^2 \tfrac{\alpha_1}{\alpha_2}\right), \\
& z_2  =\sqrt{\tfrac{\alpha_2}{\alpha_1}} \erf^{-1}(H(z_1)).
\end{array}
\right.
\end{equation}
Given that \( \mu_2 \) is defined by \eqref{mu2} and that \( \mu_1 \) is the unique solution to \eqref{QV}, we can conclude that \( \mu_1 \) and \( \mu_2 \) form a solution to the system \eqref{sistemaP1P2}. Due to the uniqueness, it follows that \( \lambda_1 = \mu_1 \) and \( \lambda_2 = \mu_2 \). Consequently, we can immediately deduce that \( \theta_i(x,t) = u_i(x,t) \) for \( i = 1, 2, 3 \).

\end{enumerate}
\end{proof}
\begin{rem} A sufficient condition for \eqref{condq-relP2} to be satisfied is
$A>\tfrac{q_2 \sqrt{\pi \alpha_3}}{k_3}+B$.
\end{rem}

\medskip

The earlier theorem reveals that problems \textbf{(P2)} and \textbf{(P3)} are equivalent under certain data-related conditions. As a result, the following relationships can be established, leading to a generalization of the findings presented in \cite{Ta1981-1982}.

\begin{cor} Let  $A>B$ be the data of the temperature boundary condition of the problem \textbf{\emph{(P2)}}. The coefficient $\mu_2$   given by \eqref{Fron-mu2} satisfies the following inequality:
\begin{equation}\label{inec-mu2-flujo}
\erf\left(\mu_2\sqrt{\tfrac{\alpha_1}{\alpha_3}}\right)< \tfrac{k_3}{k_2}\sqrt{\tfrac{\alpha_2}{\alpha_3}}\tfrac{A-B}{B-C}\erf\left( z_0\sqrt{\tfrac{\alpha_1}{\alpha_2}} \right),
\end{equation}
where $z_0$ is given by \eqref{z0}.
\end{cor}
\begin{proof}
Part b) of the previous theorem indicates that \( q_0 \), as defined by \eqref{q0(A)}, must satisfy \( q_0 > q_2 \), where \( q_2 \) is given by \eqref{q0-q2}. Consequently, the coefficient \( \mu_2 \) that characterizes the interface \( x = r_2(t) \) of the solution to problem \textbf{(P2)} must also comply with the inequality \eqref{inec-mu2-flujo}.
\end{proof}

\begin{cor} The value $A=A(q_0)$ given by \eqref{condA-relP3} is an increasing function of $q_0$. 
\end{cor}
\begin{proof}

It is important to note that \( \lambda_2 \), as defined by \eqref{lambda2}, is dependent on \( \lambda_1 \). Furthermore, \( \lambda_1 \) represents the unique solution to the equation \eqref{lambda1}, which itself is influenced by \( q_0 \). On one hand, the function \( Q \) is independent of \( q_0 \). Conversely, \( P \) is an increasing function of \( q_0 \). Consequently, \( \lambda_1 \) is also an increasing function of \( q_0 \), and this behaviour extends to \( \lambda_2 \) as well.
According to the definition of $A$, it follows directly that the thesis is validated.

\end{proof}

\begin{rem} The coefficient $q_0$ given by \eqref{q0(A)} can be thought as  a function that depends  on $A$.
If we consider  $B=328\;K$, $C=324\;K$, $D=320\;K$, 
$k_1=0.2\; \tfrac{W}{mK}$,
$k_2=0.2\; \tfrac{W}{mK}$,
$k_3=0.2 \; \tfrac{W}{mK}$,
$c_1=2  \; \tfrac{J}{kg K}$,
$c_2=2 \; \tfrac{J}{kg K}$,
$c_3=2 \; \tfrac{J}{kg K}$,
$\rho=770\; \tfrac{kg}{m^3}$,
$\ell_1=160\; \tfrac{J}{kg}$,
$\ell_2=150\; \tfrac{J}{kg}$,
we can plot \( q_0 \) as a function of \( A \).  Figure \ref{Fig2:qo(A)} clearly shows that \( q_0 \) is a strictly increasing function.

\begin{figure}[h!!!]
\begin{center}
\includegraphics[scale=0.5]{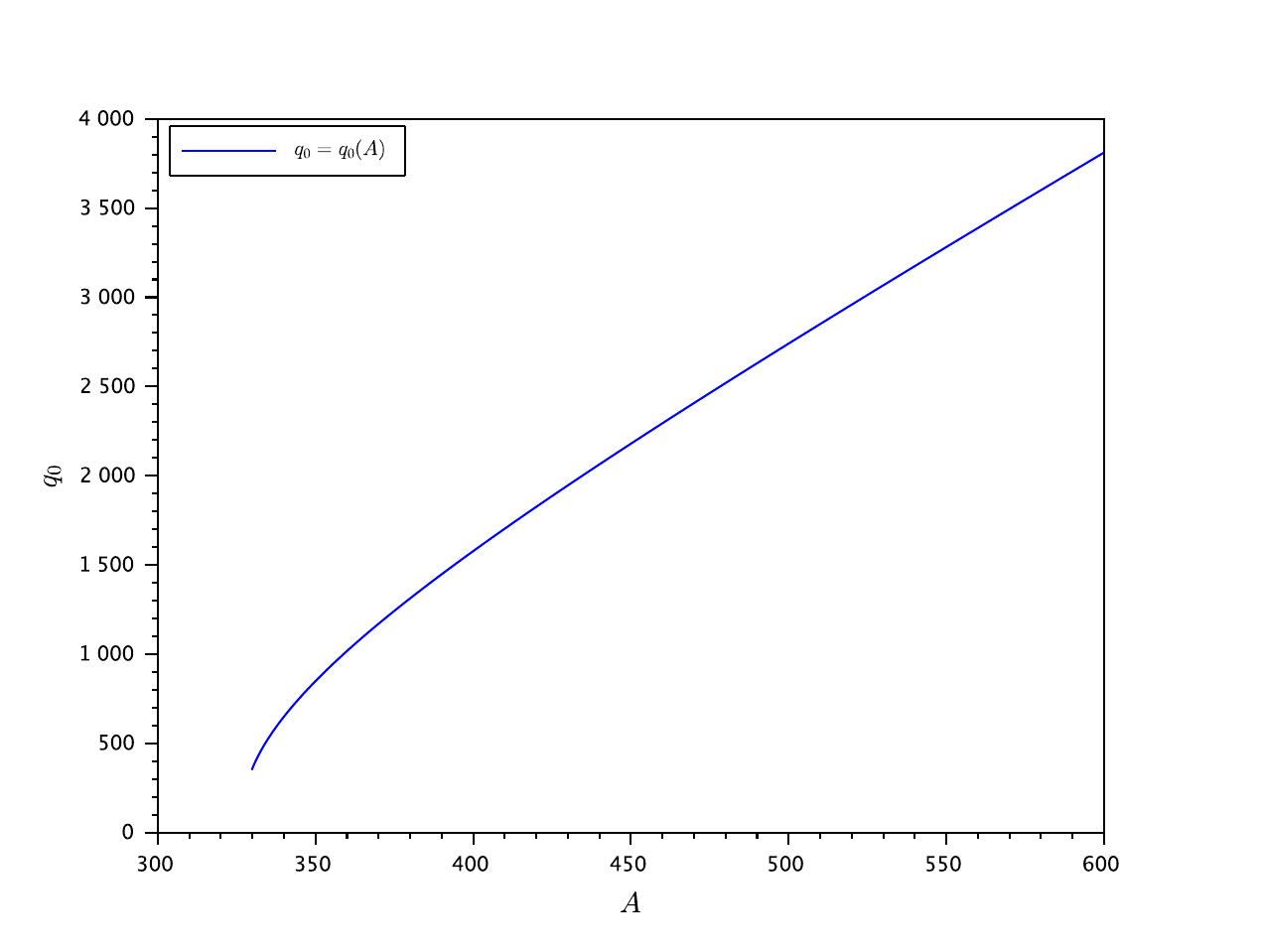}
\caption{Coefficient $q_0$ characterizing  the flux condition as a function of 
$A$.}
\label{Fig2:qo(A)}

\end{center}
\end{figure}
\end{rem}

\newpage
\subsection{Equivalence between problems \emph{\textbf{(P1)}} and \emph{\textbf{(P3)}}}

As illustrated in the previous subsections, the following theorem establishes the relationship between the problems involving convective and flux boundary conditions at the fixed face.

\begin{thm}\label{TeorP1P3} $ $
\begin{enumerate}
\item[\emph{a)}] Let $h_0$ and $A_\infty$  be  the given constants of the convective condition of the problem \emph{\textbf{(P1)}} with $A_\infty>B$ and $h_0>h_2$ where $h_2$ is given by \eqref{h2}. If the following inequality holds:
\begin{equation}\label{condq-P1P3}
\tfrac{(A_\infty-B) h_0}{1+\tfrac{h_0 \sqrt{\pi \alpha_3}}{k_3} \erf\left( \xi_2 \sqrt{\tfrac{\alpha_1}{\alpha_3}}\right)}>\tfrac{k_2(B-C)}{\sqrt{\alpha_2\pi}\erf\left(z_0\sqrt{\tfrac{\alpha_1}{\alpha_2}}\right)},
\end{equation}
where $\xi_2$ is given by \eqref{xi2}, then the solution to problem \textbf{\emph{(P3)}} with
\begin{equation}\label{q0(h)}
q_0=\tfrac{(A_\infty-B) h_0}{1+\tfrac{h_0 \sqrt{\pi \alpha_3}}{k_3} \erf\left( \xi_2 \sqrt{\tfrac{\alpha_1}{\alpha_3}}\right)},
\end{equation}
 coincides with the solution to problem \textbf{\emph{(P1)}}.

\item[\emph{b)}] Let $q_0>q_2$   be  the given constant of the flux condition of the problem \emph{\textbf{(P3)}}  where $q_2$ is defined by \eqref{q0-q2}. If the following inequality holds: 
\begin{equation}\label{condA-relP1P3}
\tfrac{q_0}{(A_\infty-B)-q_0 \tfrac{\sqrt{\pi\alpha_3}}{k_3} \erf\left(\lambda_2 \sqrt{\tfrac{\alpha_1}{\alpha_3}} \right)}>\tfrac{B-C}{A_{\infty}-B} \sqrt{\tfrac{k_2k_3c_2}{\pi \alpha_3c_3}}\tfrac{1}{\erf\left(z_0 \sqrt{\tfrac{\alpha_1}{\alpha_2}}\right)},
\end{equation}
where $\lambda_2$ is given by \eqref{lambda2} and $A_\infty>B$, 
then the solution to problem \textbf{\emph{(P1)}} with
\begin{equation}
h_0=\tfrac{q_0}{(A_\infty-B)-q_0 \tfrac{\sqrt{\pi\alpha_3}}{k_3} \erf\left(\lambda_2 \sqrt{\tfrac{\alpha_1}{\alpha_3}} \right)},
\end{equation}
coincides with the solution to problem \textbf{\emph{(P3)}}.

\end{enumerate}
\end{thm}

\begin{proof}
The proof is straightforward.
\end{proof}

\begin{rem} 
 If we assume that 
\[
A_\infty > B + \sqrt{\tfrac{\alpha_3}{\alpha_2}} \tfrac{k_2}{k_3} \tfrac{B - C}{\operatorname{erf} \left( z_0 \sqrt{\frac{\alpha_1}{\alpha_2}} \right)},
\]
and 
\[
h_0 > \max \left\{ h_2, h_2^* \right\},
\]
where \( h_2 \) is given by \eqref{h2} and \( h_2^* > 0 \) is such that \( F(h_2^*) = 0 \), with
\[
F(z) = \tfrac{k_3 (A_\infty - B) \sqrt{\pi \alpha_2} \operatorname{erf} \left( z_0 \sqrt{\frac{\alpha_1}{\alpha_2}} \right) z}{k_2 (B - C) \left( k_3 + z \sqrt{\pi \alpha_3} \right)}, \qquad z \geq 0,
\]
and \( z_0 \) is given by \eqref{z0}, then condition \eqref{condq-P1P3} is automatically satisfied.

In addition, if $$q_0>\tfrac{B-C}{A_{\infty}-B} \sqrt{\tfrac{k_2k_3c_2}{\pi \alpha_3c_3}}\tfrac{1}{\erf\left(z_0 \sqrt{\tfrac{\alpha_1}{\alpha_2}}\right)},$$ then the inequality given by \eqref{condA-relP1P3} holds, for all $A_\infty>B$.
\end{rem}

\section*{Conclusions}
 This study provided a unique explicit solution for the three-phase Stefan problem in a semi-infinite material with a convective boundary condition at the fixed face. The equivalence among the solutions of three Stefan problems with different boundary conditions (Robin, Dirichlet, and Neumann) was demonstrated, provided that a specific relationship between the problem data was satisfied. Additionally, numerical examples were performed to illustrate the validity of the obtained results and to explore the system's behavior under various boundary condition configurations. These findings offer a deeper understanding of heat transfer processes in phase-change systems, with significant implications for material science and engineering applications.

\section*{Acknowledgement}
\noindent The present work has been partially sponsored by the  project PIP-CONICET 11220220100532CO and the projects  80020210100002 and 80020210200003 from Austral University, Rosario, Argentina.

\end{document}